\documentclass[12pt,a4paper]{article}
\usepackage{amsfonts,amsmath,amssymb}
\usepackage{oldlfont}
\usepackage[cp1250]{inputenc}
\usepackage[T1]{fontenc}
\usepackage[width=17cm,height=25cm]{geometry}

\usepackage{color}

\usepackage{pgf,tikz}
\usetikzlibrary{arrows,patterns}
\usetikzlibrary{matrix,positioning}

\newtheorem{theorem}{\bf Theorem}[section]
\newtheorem{corollary}[theorem]{Corollary}
\newtheorem{lemma}[theorem]{Lemma}
\newtheorem{proposition}[theorem]{Proposition}

\newtheorem{definition}[theorem]{Definition}

\newcommand{\qed}{\hfill $\square$ \bigskip}

\begin{document}

\title{Weakly toll convexity in graph products}

\author{Polona Repolusk $^{a}$}

\maketitle

\begin{center}
$^a$ Faculty of Natural Sciences and Mathematics, University of Maribor, Slovenia\\
{\tt polona.repolusk@um.si}\\
\end{center}

\begin{abstract}

The exploration of weakly toll convexity is the focus of this investigation. A weakly toll walk is any walk $W: u, w_1, \ldots , w_{k-1}, v$ between $u$ and $v$ such that $u$ is adjacent only to the vertex $w_1$, which can appear more than once in the walk, and $v$ is adjacent only to the vertex $w_{k-1}$, which can appear more than once in the walk. Through an examination of general graphs and an analysis of weakly toll intervals in both lexicographic and (generalized) corona product graphs, precise values of the weakly toll number for these product graphs are obtained. Notably, in both instances, the weakly toll number is constrained to either 2 or 3.  Additionally, the determination of the weakly toll number for the Cartesian and the strong product graphs is established through previously established findings in toll convexity theory. Lastly for all graph products examined within our scope, the weakly toll hull number is consistently determined to be 2.

	

\end{abstract}

\section{Introduction}\label{Int}

The theory of convexity has been extensively explored, with numerous papers delving into various aspects of the subject.
Various forms of interval convexities have been thoroughly investigated, with one particularly natural aspect stemming from abstract convexity theory known as the convex geometry property or Minkowski--Krein--Milman property. In this context, a vertex $s$ of a convex set $S$  is an extreme vertex of $S$, if $S - \{s \}$ is also convex.
 A graph $G$ is deemed a convex geometry concerning a given convexity if every convex set of $G$ is the convex hull of its extreme vertices. Specifically, under monophonic convexity,  exactly chordal graphs are convex geometries \cite{stg}, while under geodesic convexity \cite{Pela}, Ptolemaic graphs (i.e., distance-hereditary chordal graphs) hold this property, as detailed in~\cite{fj-86}.

A graph convexity, for which exactly the interval graphs are convex geometry, was introduced and explored in~\cite{abg}. The authors focused on interval graphs and built upon concepts from~\cite{al-14}, which characterized this graph family in terms of tolled walks. The concept of a tolled walk, a generalization of monophonic paths, led to the definition of toll convexity \cite{abg}, which was further investigated for instance in \cite{hull,tanja}.

Expanding on this research,~\cite{weakly} established a characterization of proper interval graphs, which are convex geometry with respect to weakly toll convexity. Thus, investigating weakly toll convexity holds significance. In \cite{weakly}, the authors also investigated weakly toll number and weakly toll hull number, determining these values for trees and providing bounds for interval graphs. Additionally,  \cite{zahtevnost} explored complexity results, demonstrating that the weakly toll number of any graph can be computed in polynomial time.

Graph products have garnered significant interest over time. Concerning various types of convexities, numerous papers have been published on the subject, such as \cite{bkh,bkt,chm,CC,TP1,TP2}. In \cite{weakly}, the authors expressed interest in exploring weakly toll convexity within the framework of graph products. In our study, we describe weakly toll intervals in the lexicographic and the corona product graphs, yielding weakly toll number values. In both cases, the resulting values are either 2 or 3, depending on the second factor of the product. Despite the polynomial algorithm presented in~\cite{zahtevnost}, our findings offer precise numerical results.
The outcomes for weakly toll number in the Cartesian and the strong product graphs directly derive from results presented for toll convexity in~\cite{TP1} and~\cite{TP2}.
Additionally, we find that for all graph products examined within our scope, the weakly toll hull number is consistently determined to be 2.

\section{Preliminaries}\label{Pre}

All graphs, considered in this paper, are finite, simple, non trivial (i.e. graphs with at least two vertices), connected and without multiple edges or loops.

Let $G=(V(G),E(G))$ be a graph and $u \in V(G)$. A set $N_G(u)=\lbrace v \in V(G) ~|~ uv \in E(G) \rbrace$ is an open neighbourhood of $u$ and $N_G\left[u\right]=N_G(u) \cup \lbrace u \rbrace$ is a closed neighbourhood of $u$. If the graph is clear from the context, we will omit the subscriptions.

A tolled walk \cite{abg} is any walk $T: u_0, u_1,\ldots. , u_{k-1}, u_k$ such that $u_0$ is adjacent only to the second vertex of
the walk, and $u_k$ is adjacent only to the second-to-last vertex of the walk. This implies that each of $u_1$ and $u_{k-1}$ occurs exactly once in the walk. 
Let $T_G(u, v) = \lbrace x \in V (G) ~|~ x \text{ lies in a tolled walk between } u \text{ and } v\rbrace$ be the toll interval between $u$ and $v$ in $G$. A subset $S$ of $V (G)$ is toll convex if $T_G(u, v) \subseteq S$ for all $u, v \in S.$   The toll closure of a subset is the union of toll intervals between all pairs of vertices from $S$, i.e., $T_G(S) = \cup_{u,v \in S}T_G(u, v).$
If $T_G(S) = V (G)$, $S$ is called a toll set of $G$. The order of a minimum toll set in $G$ is called the toll number of $G$ $(tn(G))$.

Let us recall the definition of the weakly toll walks from \cite{weakly}: 
Let $u,v \in V(G)$. A {\it weakly toll walk} $W$ between $u$ and $v$ in $G$ is a sequence of vertices of the form $$W: u=w_0,w_1,\ldots,w_{k-1},w_k=v,$$ where $k\ge 0$, which enjoys the following three conditions for $k>0$:
\begin{itemize}
\item $w_iw_{i+1}\in E(G)$ for all $i \in \lbrace 0,\ldots,k-1\rbrace$,
\item $uw_i\in E(G)$ implies $w_i=w_1$, where $i\in \lbrace 1,\ldots , k \rbrace$,
\item $w_i v\in E(G)$ implies $w_i=w_{k-1}$, where $i\in \lbrace 0,\ldots , k-1 \rbrace$.
\end{itemize}
In other words, a weakly toll walk is any walk $W: u, w_1, \ldots , w_{k-1}, v$ between $u$ and $v$ such that $u$ is adjacent only to the vertex $w_1$, which can appear more than once in the walk, and $v$ is adjacent only to the vertex $w_{k-1}$, which can appear more than once in the walk. If $uv \in E(G)$, then $W: u, v$ is a weakly toll walk, and if $k = 0$ then $W: u$ is a wealky toll walk consisting of a single vertex.

We define $WT_G(u, v) = \lbrace x \in V (G) ~|~ x \text{ lies in a weakly toll walk between } u \text{ and } v\rbrace$ to be the weakly toll interval between $u$ and $v$ in $G$. Finally, a subset $S$ of $V (G)$ is weakly toll convex if $WT_G(u, v) \subseteq S$ for all $u, v \in S$.
Note that any weakly toll convex set is also a toll convex set. Also, any toll convex set is a monophonically convex set, and a monophonically convex set is a geodesically convex set.
On the other hand, consider the graph $K_{1,3}$ with vertices $a, b, c, d$, where $b$ is the vertex with degree three, and let $S = \lbrace a, b, c \rbrace$. It is clear that $S$ is toll convex but not weakly toll convex, since $W: a, b, d, b, c$ is a weakly toll walk between $a$ and $c$ that contains $d \notin S$ \cite{weakly}.

The definition of weakly toll interval for two vertices $u$ and $v$ can be generalized to an arbitrary subset $S$ of $V (G)$ as follows:
$$WT_G(S) = \cup_{u,v \in S}WT_G(u, v).$$
If $WT_G(S) = V (G)$, $S$ is called a weakly toll set of $G$. The order of a minimum weakly toll set in $G$ is called the weakly toll number of $G$, and is denoted by $wtn(G)$. In \cite{zahtevnost}, this concept is defined as weakly toll interval number.
For any non trivial connected graph $G$, it is clear that $2 \leq wtn(G) \leq n.$

The weakly toll convex hull of a set $S \subseteq V (G)$ is defined as the intersection of all weakly toll convex sets that contain $S$, and we will denote this set by $\left[ S \right]]_{WT}$. A set $S$ is a weakly toll hull set of $G$ if  $\left[ S \right]]_{WT}=V(G)$. The weakly toll hull number of $G$, denoted by $wth(G)$, is the minimum among all the cardinalities of weakly toll hull sets. Given a set $S \subseteq V (G)$, define $WT^k(S)$ as follows: $WT^0(S) = S$ and $WT^k(S) = WT (WT^{ k-1}(S)),$
for $k \geq 1$. Note that  $\left[ S \right]]_{WT}=\cup_{k \in {\mathbb N}}WT^k(S)$. From the definitions we immediately infer that every weakly toll set is a weakly toll hull set, and hence $wth(G) \leq wtn(G)$.

In this paper we will only consider non complete graphs. Note that $wtn(K_n)=n$. It is clear from the definitions that $wtn(G) \leq tn(G)$. Also, $wtn(T)=2$ for any tree $T$ as weakly toll interval between any two leaves is a whole vertex set. Similarly, assume that a graph $G$ has two vertices of degree 1. Then also $wtn(G)=2$. 
Note that there are graphs with $wtn(G)>2$. For instance, take two copies of a complete graph $K_n$. Now, take one vertex of one copy of $K_n$ and join it with one vertex of another copy of $K_n$ by a path of length 2. Then $wtn(G)=2n-2.$

We follow the definitions of graph products by the book \cite{ImKl}.
Recall that for all of the standard graph products, the vertex set of the product of graphs $G$ and $H$ is equal to $V(G)\times V(H)$.
In the lexicographic product $G\left[H\right]$, vertices $(g_{1},h_{1})$ and $(g_{2},h_{2})$ are adjacent if either $g_{1}g_{2}\in E(G)$ or ($g_{1}=g_{2}$ and $h_{1}h_{2}\in E(H)$). 
In the Cartesian product $G\Box H$ of graphs $G$ and $H$, two vertices $(g_{1},h_{1})$ and $(g_{2},h_{2})$ are adjacent when ($g_1g_2\in E(G)$ and $h_1=h_2$) or ($g_1=g_2$ and $h_1h_2\in E(H)$). 
In the strong product $G\boxtimes H$ of graphs $G$ and $H$, two vertices $(g_{1},h_{1})$ and $(g_{2},h_{2})$ are adjacent  if either ($g_1g_2\in E(G)$ and $h_1=h_2$) or ($g_1=g_2$ and $h_1h_2\in E(H)$) or ($g_1g_2 \in E(G)$ and $h_1h_2 \in E(H)$).

Let $G$ and $H$ be graphs and $*$ be one of the two graph products under consideration. For a vertex $h\in V(H)$, we call the set $G^{h}=\{(g,h)\in V(G * H):g\in
V(G)\}$ a $G$-\emph{layer} of $G * H$. By abuse of notation we also consider $G^{h}$ as the corresponding induced subgraph. Clearly $G^{h}$ is isomorphic to $G$. For $g\in V(G)$, the $H$-\emph{layer} $^g\!H$ is defined as $^g\!H =\{(g,h)\in V(G * H)\,:\,h\in V(H)\}$. We also consider $^g\!H$ as an induced subgraph and note that it is isomorphic to $H$. A map $p_{G}:V(G * H)\rightarrow V(G)$, $p_{G}(g,h) = g$ is the \emph{projection} onto $G$ and $p_{H}:V(G * H)\rightarrow V(H)$, $p_{H}(g,h) = h$ the \emph{projection} onto $H$.

The corona product of two graphs $G$ and $H$, $G \circ H$, is defined as the graph obtained by taking one copy of $G$ and $|V(G)|$ copies of $H$ and joining the $i$-th vertex of $G$ to every vertex in the $i$-th copy of $H$. For more details on graph products see~\cite{ImKl}. 

In \cite{TP1}, authors derived that $tn(G\Box H)=2$ for any connnected non trivial graphs $G$ and $H$. This leads to:

\begin{corollary}
Let $G$ and $H$ be two connected non trivial graphs. Then $wtn(G \Box H)=2.$
\end{corollary}
They also found a formula for computing $tn(G \left[H \right])$ in terms of the so-called toll-dominating triples and some general bounds for this number. In \cite{TP2}, the toll number of the strong product was determined as $tn(G \boxtimes H) \leq 3$ with the characterization for both cases. Therefore:

\begin{corollary}
Let $G$ and $H$ be two connected non complete graphs. Then $wtn(G \boxtimes H)\leq 3.$
\end{corollary}

In the subsequent sections, we will initially describe the weakly toll intervals between non adjacent vetrices in the lexicographic product graphs, outlined in Section \ref{s:Lex}. Subsequently, we will ascertain the weakly toll number for this scenario. Section \ref{s:Corona} will be dedicated to describing weakly toll intervals between non adjacent vertices of the corona product of graphs. Once again, we will demonstrate the precise values of $wtn(G \circ H)$. Concluding this paper, we will offer remarks on the generalized corona product of graphs.

However, before delving further, we present two highly pertinent results that will greatly assist in our subsequent analysis:

\begin{lemma}\label{sosedi}
Let $G$ be a connected non complete graph and $u,v$ two non adjacent vertices of $G$. Let $W=WT_G(u,v)$ be a weakly toll interval between vertices $u$ and $v$ and $x \in V(G) \setminus (N\left[u\right] \cup N\left[v\right])$ such that $xy \in E(G)$ for some $y \in W\setminus \lbrace u,v \rbrace$. Then $x \in W$.
\end{lemma}

\begin{proof}
Assume $xy \in E(G)$ for some $y \in W\setminus \lbrace u,v \rbrace.$ Let $W_1$ be a weakly toll walk between $u$ and $v$ which contains $y$, say $W_1:u,w_0,w_1,\ldots,w_{i-1},y,w_{i+1},\ldots,w_n,v$. Then $W_1':u,w_0,w_1,\ldots,w_{i-1},y,x,y,w_{i+1},\ldots,w_n,v$ is a weakly toll walk containing $x$. Note that this also holds if $y=w_0$ or $y=w_n$ as  $xu \notin V(G)$ and $xv \notin V(G)$.\qed
\end{proof}

Combining this result with some propositions, that were used in the proof of the main theorem in \cite{zahtevnost}, we can derive the following:

\begin{lemma}\label{clanek2}
Let $G$ be a connected non complete graph and $u,v$ two non adjacent vertices of $G$. Let $W=WT_G(u,v)$ be a weakly toll interval between vertices $u$ and $v$. Assume that $W$ is such that  $|WT_G(u, v)| \geq |WT_G(w, z)|$ for any $w, z  \in V (G).$ 	Let $X_u=N\left[u\right]  \setminus W$, $X_v=N\left[v\right]  \setminus W$ and $X=V(G)\setminus W$. Then $X_u \cap X_v = \emptyset$ and $X=X_u \cup X_v$.
\end{lemma}

\begin{proof}
First, assume that $X= \emptyset$. Then $W=V(G)$, $X_u=\emptyset$, $X_v=\emptyset$ and $\lbrace u,v \rbrace$ is a weakly toll set.
Now, let $X = V(G)\setminus W$ and $x \in X$. If $xy \in E(G)$ for some $y \in W\setminus \lbrace u,v \rbrace$. Then, by Lemma \ref{sosedi}, $x \in W$, a contradiction.
To complete the proof, let $x \in X \setminus( X_u \cup X_v)$,  such that for every $y \in W$, $xy \notin E(G)$. Every shortest path from $u$ to $x$ (resp. $v$ to $x$) contains:
\begin{enumerate}
	\item exactly one vetrex from $X_u$ (resp.$X_v$). If $uw \in E(G)$, where $w \in N(u)\cap W$, Lemma \ref{sosedi} would imply that $x \in W$.
	 \item no vertex of $N\left[v \right]$ (resp.$N\left[u \right]$). If it contains $v$, then this path only between $u$ and $v$ is a weakly toll walk containing a vertex of $X_u$, a contradiction. If it contains a neighbor of $v$, say $w$, then if $w \in W$, by Lemma \ref{sosedi}, $x \in W$, a contradiction. Otherwise, if $w \in X_v$, observe that $v \in WT_G(x,u)$: Let $P:u,w_1,\ldots,w_i,w,w_{i+1},\ldots,w_k,x$ be a shortest path between $u$ and $x$ going through $w$. As $uv \notin E(G)$ and $xv \notin E(G)$, $P':u,w_1,\ldots,w_i,w,v,w,w_{i+1},\ldots,w_k,x$ is a weakly toll walk between $u$ and $x$, containing $v$. Therefore, $WT_G(x,u) \supset WT_G(u,v)$, a contradiction.
	\end{enumerate}
To summarize, there are paths $P$ between $u$ and $x$ containing no vertex of $N\left[v \right]$ and $Q$ between $x$ and $v$ containing no vertex of $N\left[u \right]$. A concatenation of paths $P$ and $Q$ is a weakly toll walk from $u$ to $v$ containing $x$, a contradiction.
Therefore, $X=X_u \cup X_v$. Assume that there is $z \in X_u \cap X_v$. Then $W: u,z,v$ is a weakly toll walk and $z \in W$, a contradiction. \qed
\end{proof}

We say that a weakly toll interval between vertices $u$ and $v$ in a graph, for which $|WT_G(u, v)| \geq |WT_G(w, z)|$ for any $w, z  \in V (G)$, is maximum weakly toll interval of $G$.

\begin{corollary}\label{glavnaPOSL}
Let $G$ be a connected non complete graph. Then $wtn(G)>2$ if and only if it holds that for any two non-adjacent vertices $u,v \in V(G)$ such that  $|WT_G(u, v)|$ is maximum, $X_u \cup X_v\neq \emptyset$. 
\end{corollary}

\begin{proof}
Let $G$ be a connected non complete graph with $wtn(G)=2$ and let $u,v \in V(G)$ be such that $WT_G(u,v)$ is maximum. Therefore, by Lemma \ref{clanek2}, $X=X_u \cup X_v$. On the other hand, as  $WT_G(u,v)=V(G)$, it follows that $X=V(G) \setminus WT_G(u,v)=\emptyset$, therefore $X_u \cup X_v=\emptyset.$
 
Now let $wtn(G)>2$. Then, for every $z,w \in V(G)$, $WT_G(z,w)\neq V(G)$. Let $u,v \in V(G)$ such that $WT_G(u,v)$ is maximum and let $x \in V(G) \setminus WT_G(u,v)$. Then, by Lemma \ref{clanek2}, $x\in X_u \cup X_v$.\qed
\end{proof}

Assume that we can describe weakly toll intervals between any two vertices of a graph. Then Corollary \ref{glavnaPOSL} implies that $wtn(G)=2$ if and only if there is a maximum weakly toll set of $G$ with $X_u \cup X_v=\emptyset$.

\section{The lexicographic product of graphs}\label{s:Lex}

Let us first describe weakly toll intervals between non adjacent vertices in the lexicographic product of two connected graphs $G$ and $H$.

\begin{lemma}\label{lexLema1}
Let $G$ and $H$ be connected non complete graphs and consider a lexicographic product graph $G\left[H\right]$. 
Let $g \in V(G)$ and let $h_1,h_2 \in V(H)$ be two non adjacent vertices.
Then  $WT_{G\left[H\right]}((g,h_1),(g,h_2))=V(G\left[H\right]) \setminus X$, 
where $$X=\lbrace (g,x) \in V( G\left[H\right]) ~|~ x\notin WT_{H}(h_1,h_2) \text{ and } x \text{ is adjacent to exactly one of } h_1 \text{ and } h_2\rbrace.$$
\end{lemma}

\begin{proof}
Obvioulsly, if $h \in WT_H(h_1,h_2)$, then $(g,h) \in WT_{G\left[H\right]}((g,h_1),(g,h_2))$ for any $g \in V(G)$. Let $g' \in V(G)$ be a neighbor of $g$, i.e. $gg' \in E(G)$. Then $W:(g,h_1),(g',h),(g,h_2)$ is a weakly toll walk for any $h \in V(H)$.
Now, let $g'' \in V(G)$ such that $gg'' \notin E(G)$. Let also $P:g,g_1,\ldots,g_k,g''$ be a shortest path between $g$ and $g''$ in $G$. Then  $$W:(g,h_1),(g_1,h_1),\ldots,(g_k,h_1),(g'',h),(g_k,h_1),\ldots,(g_1,h_1),(g_1,h_2)$$ is a weakly toll walk between $(g,h_1)$ and $(g,h_2)$, which contains a vertex $(g'',h)$ for any $h \in V(H)$.

To conclude the proof, let $(g,x) \in V( G\left[H\right])$ be such that  $x\notin WT_{H}(h_1,h_2)$. Let us consider three cases:
\begin{enumerate}
	\item If $x h_1 \notin E(H)$ and $x h_2 \notin E(H)$. Then for some neighbor $g'$ of $g$, $$W:(g,h_1),(g',h_1),(g,x),(g',h_1),(g,h_2)$$ is a weakly toll walk. 
	\item If $x h_1 \in E(H)$ and $x h_2 \notin E(H)$, as $(g,x) \notin WT_{H^g}((g,h_1),(g,h_2))$, then if $W$ is a weakly toll walk between $(g,h_1)$ and $(g,h_2)$ containing a vertex $(g,x)$, $W$ must contain some vertex in $V(^{g'}\!H)$, where $gg' \in E(G)$.  As $(g,h_1)(g',h) \in E(G\left[H\right])$ for any $h \in V(H)$, this is a contradiction with $W$ being a weakly toll walk.
	\item  Similarly, if $x h_1 \notin E(H)$ and $x h_2 \in E(H)$, then as $(g,h_2)(g',h) \in E(G\left[H\right])$ for any $h \in V(H)$, there is no weakly toll walk between $(g,h_1)$ and $(g,h_2)$ which would contain $(g,x)$.
\end{enumerate} 
Note that if $x h_1 \in E(H)$ and $x h_2 \in E(H)$, then $x\in WT_{H}(h_1,h_2)$. \qed
\end{proof}

\begin{lemma}\label{lexLema2}
Let $G$ and $H$ be connected non complete graphs and consider a lexicographic product graph $G\left[H\right]$. 
Let $g_1,g_2 \in V(G)$ be non adjacent vertices and let $h_1,h_2 \in V(H)$ be non adjacent vertices. Then  $WT_{G\left[H\right]}((g_1,h_1),(g_2,h_2))=\lbrace (g,h) ~|~ g \in WT_G(g_1,g_2), h \in V(H)\rbrace \setminus X$, where $X=N_{^{g_1}\!H}((g_1,h_1)) \cup N_{^{g_2}\!H}((g_2,h_2)).$
\end{lemma}

\begin{proof}
Let $g_1,g_2 \in V(G)$ be non adjacent vertices and let $h_1,h_2 \in V(H)$ be non adjacent vertices. First, let $g \in WT_G(g_1,g_2)\setminus \lbrace g_1,g_2\rbrace$ and let $W:g_1,g_1',\ldots,g_k',g,g_{1}''\ldots,g_l'',g_2$ be a weakly toll walk in $G$. Then $$W_1:(g_1,h_1),(g_1',h_1), \ldots,(g_k',h_1),(g,h),(g_{1}'',h_2),\ldots,(g_l'',h_2),(g_2,h_2)$$ is a weakly toll walk containing $(g,h)$ for any $h \in V(H)$.

Let now $h \in V(H) \setminus N\left[h_1\right]$ and $P:g_1,g_1',\ldots g_k',g_2$ a path in $V(G)$. Then
$$W:(g_1,h_1),(g_1',h_1),(g_1,h),(g_1',h_1),(g_2',h_1),\ldots,(g_k',h_1),(g_2,h_2)$$ is a weakly toll walk in $G \left[ H \right]$ containing a vertex $(g_1,h)$. Similarly we get a weakly toll walk containing a vertex $(g_2,h')$, where $h' \in V(H) \setminus N\left[h_2\right]$.

Let $(g_1,h) \in V(G\left[H\right])$ and let $hh_1 \in E(H)$. Note that any weakly toll walk between $(g_1,h_1)$ and $(g_2,h_2)$ contains a vertex $(g',h')$, where $g_1g'\in E(G)$ and $h' \in V(H)$. As $(g_1,h_1)(g_1,h) \in E(G\left[H\right])$ and also $(g_1,h_1)(g',h') \in E(G\left[H\right])$ for every $h' \in V(H)$, a weakly toll walk between $(g_1,h_1)$ and $(g_2,h_2)$ cannot contain a vertex $(g_1,h)$. Symetrically it follows that no vertex of $N((g_2,h_2)) \cap ^{g_2}\!H$ is in $WT_{G\left[H\right]}((g_1,h_1),(g_2,h_2))$.

To conclude the proof, let $(g,h) \in V(G\left[H\right])$, where $g \notin WT_G(g_1,g_2)$ and $h \in V(H)$. 
Suppose that $(g,h) \in  WT_{G\left[H\right]}((g_1,h_1),(g_2,h_2))$ and let $$W_2:(g_1,h_1),(g_1',h_1'), \ldots,(g_k',h_k'),(g,h),(g_{1}'',h_1''),\ldots,(g_l'',h_l''),(g_2,h_2)$$ be a weakly toll walk containing $(g,h)$. Recall that $(g_1',h_1')$ is the only neighbor of $(g_1,h_1)$ and $(g_l'',h_l'')$ is the only neighbor of $(g_2,h_2)$ in $W_2$. Therefore $g_1\neq g_1'$ and $g_2 \neq g_l''$, moreover,  $(N(g_1)\times V(H)) \cap W_2=\lbrace (g_1',h_1') \rbrace$ and  $(N(g_2)\times V(H)) \cap W_2=\lbrace (g_l'',h_l'') \rbrace$. 
Assume that for some $x \in V(H)$, $(g_1,x) \in W_2$. Then $xh_1 \notin E(H)$ and for every such vertex $(g_1,x)$, there is a subwalk $S_1: (g_1',h_1'),(g_1,x),(g_1',h_1')$ in $W$. Similarly, assume that for some $y \in V(H)$, $(g_2,y) \in W$. Then $yh_2 \notin E(H)$ and for every such vertex $(g_2,y)$, there is a subwalk $S_2: (g_l'',h_l''),(g_2,y),(g_l'',h_l'')$ in $W_2$. Let $W_2'$ be a walk in $G\left[H \right]$, in which we replace each subwalk $S_1$ with $(g_1',h_1')$ and each subwalk $S_2$ with $(g_l'',h_l'')$. Then the projection of $W_2'$ is a weakly toll walk in $G$ between $g_1$ and $g_2$, containing a vertex $g$, which is a contradiction.\qed
\end{proof}

\begin{corollary}\label{lex2}
Let $G$ and $H$ be connected non complete graphs with $wtn(H)=2$. Then $wtn(G \left[H\right])=2.$
\end{corollary}

\begin{proof}
Let $W=\lbrace h_1,h_2 \rbrace$ be a weakly toll set of $H$ and $g \in V(G)$. By Lemma \ref{lexLema1}, a set $\lbrace (g,h_1),(g,h_2) \rbrace$  is a weakly toll set of $G \left[H\right]$.
\qed
\end{proof}

\begin{proposition}\label{lex3}
Let $G$ and $H$ be connected non complete graphs. Then $wtn(G \left[H\right])\leq 3.$
\end{proposition}

\begin{proof}
Let $g_1,g_2 \in V(G)$ and $d_G(g_1,g_2)=2$ and denote with $g \in V(G)$ a vertex, for which $gg_1\in E(G)$ and $gg_2 \in E(G)$. We have such vertices as $G$ is not a complete graph. Let $h_1,h_2 \in V(H)$ be two non adjacent vertices. By Lemma \ref{lexLema1}, $WT_{G \left[H\right]}((g_1,h_1),(g_1,h_2))=V(G\left[H\right]) \setminus X$, 
where $$X=\lbrace (g_1,x) \in V( G\left[H\right]) ~|~ x\notin WT_{H}(h_1,h_2) \text{ and } x \text{ is adjacent to exactly one of } h_1 \text{ and } h_2\rbrace.$$
Let now $(g_1,x) \in V( G\left[H\right])$ be such a vertex, that $x\notin WT_{H}(h_1,h_2)$ and  $(g_1,x)(g_1,h_1) \in E(G\left[H\right])$ but  $(g_1,x)(g_1,h_2) \notin E(G\left[H\right])$. 
Then $W:(g_1,h_2),(g,h_2),(g_1,x),(g,h_2),(g_2,h_1)$ is a weakly toll walk between $(g_1,h_2)$ and $(g_2,h_1)$ containing $(g_1,x)$. If $(g_1,x)(g_1,h_1) \notin E(G\left[H\right])$ and $(g_1,x)(g_1,h_2) \in E(G\left[H\right])$, then $W:(g_1,h_1),(g,h_1),(g_1,x),(g,h_1),(g_2,h_1)$ is a weakly toll walk between $(g_1,h_1)$ and $(g_2,h_1)$ containing $(g_1,x)$. Therefore, a set $\lbrace (g_1,h_1),(g_1,h_2),(g_2,h_1)\rbrace$ is a weakly toll set of $G \left[H\right]$. \qed
\end{proof}

\begin{theorem}\label{glavniLex}
Let $G$ and $H$ be connected non complete graphs.
Then 
\begin{equation*}
wtn(G \left[H\right]) = \left\{
\begin{array}{cl}
2; & wtn(H)=2,\\
3; & \text{otherwise.}
\end{array} \right.
\end{equation*}	
\end{theorem}

\begin{proof}
By Proposition  \ref{lex3}, $wtn(G \left[H\right]) \leq 3$. Investigating Lemma \ref{lexLema1} and Lemma \ref{lexLema2} we observe that for given graphs $G$ and $H$, maximum weakly toll intervals arrise when the interval is between two non adjacent vertices of one copy of $H$. Such an interval is maximum iff a set
$$X=\lbrace x \in V(H) ~|~ x\notin WT_{H}(h_1,h_2) \text{ and } x \text{ is adjacent to exactly one of } h_1 \text{ and } h_2\rbrace$$ is minimum, i.e. an empty set.  By the notation, used in Lemma \ref{clanek2}, observe that $X=X_{h_1} \cup X_{h_2}$. Moreover, applying Corolarry \ref{glavnaPOSL} leads to $X_{h_1} \cup X_{h_2}=\emptyset$ if and only if $wtn(H)=2.$ \qed
\end{proof}

Even more:
\begin{theorem}\label{hullLex}
Let $G$ and $H$ be connected non complete graphs. Then $wth(G \left[H\right]) = 2.$
\end{theorem}

\begin{proof}
Note that $wth(G\left[H\right])\leq wtn(G\left[H\right]) \leq 3.$ Recall the proof of Proposition \ref{lex3}. Let $S=\lbrace (g_1,h_1),(g_1,h_2) \rbrace$. Then  $\left[ S \right]]_{WT}=V(G\left[H\right])$ as the vertex $(g_2,h_1) \in WT((g_1,h_1),(g_1,h_2))$.\qed
\end{proof}

\section{The Corona product of graphs}\label{s:Corona}

Let $G$ and $H$ be connected non complete graphs and denote $|V(G)|=n,$ $V(G)= \lbrace g_1,\ldots,g_n \rbrace$, $|V(H)|=m$ and $V(H)= \lbrace h_1,\ldots,h_m \rbrace$. Let $H^i$ denote the $i$-th copy of $H$ for all $i \in \lbrace 1,\ldots,n\rbrace$ and $V(H^i)=\lbrace h_1^i,\ldots, h_m^i \rbrace$. We define the corona produch of graphs $G$ and $H$, $G\circ H$, in the following way: $V(G \circ H)=V(G) \cup \bigcup_{j=1}^mV(H^j)$ and
$$E(G \circ H)= E(G) \cup \bigcup_{i=1}^n\left(E(H^i) \cup  \bigcup_{j=1}^m g_ih_j^i\right).$$
For every $i \in \lbrace 1,\ldots,n\rbrace$,  $E(H^i)$ represents the copy of the set $E(H)$.

Let us describe weakly toll intervals between vertices of $G \circ H$. First we begin with the case, where both vertices are of the same copy of $H$.

\begin{lemma}\label{lema1}
Let $G$ and $H$ be connected non complete graphs and consider a corona product graph $G \circ H$. Let $i \in \lbrace 1,\ldots,n \rbrace$ and let $h_1^i,h_2^i \in V(H^i)$ be  arbitrary non adjacent vertices of the $i$-th copy of $H$ in $G\circ H$. Then  $WT_{G\circ H}(h_1^i,h_2^i)=V(G \circ H) \setminus X$,  where $$X=\lbrace x \in V(H^i) ~|~ x\notin WT_{H^i}(h_1^i,h_2^i) \text{ and } x \text{ is adjacent to exactly one of } h_1^i \text{ and } h_2^i\rbrace.$$
\end{lemma}

\begin{proof}
W.l.o.g., let $i=1.$ Let $h_1^1,h_2^1 \in V(H^1)$ be  arbitrary non adjacent vertices of $H^1$. First, observe that if $x \in V(G) \cup \bigcup_{k=2}^nV(H^k)$, then $x \in WT_{G\circ H}(h_1^1,h_2^1)$.
 As $W: h_1^1,g_1,h_2^1$ is a weakly toll walk, $g_1 \in WT_{G\circ H}(h_1^1,h_2^1)$. Note that if $x \neq g_1,$ then $x$ is not adjacent to either $h_1^1$ nor $h_2^1$. Consider a vertex $g_j$ for any $j \in \lbrace 2,3,\ldots,n\rbrace$. As $G$ is connected, there is a path, say $P_j$ between $g_1$ and $g_j$ in $G$. Let also $Q$ be a walk in $H^j$ that contains all vertices of $H^j$ and $\tilde{P_j}$ a path in $G$ between $g_j$ and $g_1.$ Then 
$$W: h_1^1,P_j,Q,\tilde{P_j},h_2^1,$$
is a weakly toll walk containing $g_j$ and every vertex of $V(H^j)$ for every $j \in \lbrace 2,3,\ldots,n\rbrace$. Therefore, $V(G) \cup \bigcup_{k=2}^nV(H^k) \subseteq WT_{G\circ H}(h_1^1,h_2^1).$

Now, let $x \in V(H^1) \setminus \lbrace h_1^1,h_2^1\rbrace$. Let us consider three cases:
\begin{enumerate}
	\item If $xh_1^1 \in E(H^1)$ and $xh_2^1 \in E(H^1)$, then $W:h_1^1,x,h_2^1$ is a weakly toll walk and $x \in WT_{G \circ H}(h_1^1,h_2^1)$.
	\item If $xh_1^1 \notin E(H^1)$ and $xh_2^1 \notin E(H^1)$, then $W:h_1^1,g_1,x,g_1,h_2^1$ is a weakly toll walk and $x \in WT_{G \circ H}(h_1^1,h_2^1)$.
	\item  Let $xh_1^1 \in E(H^1)$ and $xh_2^1 \notin E(H^1)$.  If $x \in WT_{H^1}(h_1^1,h_2^1)$, then $x \in WT_{G \circ H}(h_1^1,h_2^1)$.  To conclude the proof, assume that $x \notin WT_{H^1}(h_1^1,h_2^1)$, but $x \in WT_{G \circ H}(h_1^1,h_2^1)$. Therefore, a weakly toll walk between $h_1^1$ and $,h_2^1$ should contain $g_1$, but as $g_1h_1^1 \in E(G \circ H)$  and $xh_1^1 \in E(G \circ H)$, this is a contradiction. Therefore, in this case, $x \notin WT_{G \circ H}(h_1^1,h_2^1)$.
Note that the case $xh_1^1 \notin E(H^1)$ and $xh_2^1 \in E(H^1)$ is very similar.
\end{enumerate}\qed
\end{proof}

Let us continue with the description of weakly toll intervals between vertices of $G \circ H$, which lie in different copies of $H$.

\begin{lemma}\label{lema2}
Let $G$ and $H$ be connected non complete graphs and consider a corona product graph $G \circ H$. Let $i,j \in \lbrace 1,\ldots,n\rbrace$ be distinct and let $k,l \in \lbrace 1,\ldots,m\rbrace$. Consider vertices $h_k^i \in V(H^i)$ and $h_l^j \in V(H^j)$. Then  $WT_{G\circ H}(h_k^i,h_l^j)=V(G \circ H) \setminus \left( N_{H^i}(h_k^i) \cup N_{H^j}(h_l^j)  \right).$
\end{lemma}

\begin{proof}
W.l.o.g, let $i=1$ and $j=2$. Let $P$ be a path between $g_1$ and $g_2$ in $G$. Then $W: h_k^1,P,h_l^2$ is a weakly toll walk. Therefore $g_1,g_2 \in WT_{G\circ H}(h_k^1,h_l^2)$.

Let $g_o \in V(G) \setminus \lbrace g_1,g_2 \rbrace$. Let $P_{g_1,g_o}$ be a path between $g_1$ and $g_o$ in $G$ and let $P_{g_o,g_2}$ be a path between $g_o$ and $g_2$ in $G$. Note that concatenation of paths $P_{g_1,g_o}$ and $P_{g_o,g_2}$ is a walk and $g_1$ and $g_2$ can appear more than once in this walk.
Let also $Q$ be a walk in $H$, which contains every vertex of $V(H)$ and $Q^o$ its copy in $V(H^o)$.
Then $W:h_k^1,P_{g_1,g_o},Q^o,P_{g_o,g_2},h_l^2$ is a weakly toll walk containing $g_o$ and every vertex of a set $V(H^o)$. 
Therefore, $\bigcup_{i=1}^n\lbrace g_i\rbrace \cup\bigcup_{i=3}^n  V(H^i) \subseteq WT_{G\circ H}(h_k^1,h_l^2)$.

Let $x \in V(H^1)$, $x \notin N\left[h_k^1\right]$ and $P$ a path between vertices $g_1$ and $g_2$ in $G$. Then $W: h_k^1,g_1,x,P,h_l^2$ is a weakly toll walk in $G \circ H$ containing $x$. Similarly, Let $y \in V(H^2)$, $y \notin N\left[h_l^2\right]$. Then $W: h_k^1,P,y,g_2,h_l^2$ is a weakly toll walk in $G \circ H$ containing $y$.

To conclude the proof, let $x \in N_{H^1}(h_k^1)$. Note that $g_1$ lies on any weakly toll walk between $h_k^1$ and $h_l^2$. As $h_k^1g_1,h_k^1x \in E(G \circ H)$, $x \notin WT_{G\circ H}(h_k^i,h_l^j)$. Similarly, if  $y \in N_{H^2}(h_l^2)$, $y \notin WT_{G\circ H}(h_k^i,h_l^j)$ as  $h_l^2g_2,h_l^2y \in E(G \circ H)$.\qed
\end{proof}

Before we  consider weakly toll intervals in $G \circ H$ between vertices of $G$, let us present an observation about weakly toll walks in the corona product graphs:

\begin{proposition}\label{sprehodi}
Let $G$ and $H$ be connected non complete graphs and consider a corona product graph $G \circ H$. Let $g_1,g_2,g \in V(G)$, where $g_1g_2 \notin E(G)$. Then $g \in WT_{G}(g_1,g_2)$ if and only if $g \in WT_{G\circ H}(g_1,g_2)$.
\end{proposition}

\begin{proof}
If $g \in WT_{G}(g_1,g_2)$, then obviously $g \in WT_{G\circ H}(g_1,g_2)$. 
Assume now that $g \notin WT_{G}(g_1,g_2)$ and suppose that  $g \in WT_{G\circ H}(g_1,g_2)$. Denote with $W$ a weakly toll walk between $g_1$ and $g_2$ in $G \circ H$ that contains $g$. Then $W$ must contain some vertex from the set $V(G \circ H) \setminus V(G)$. Now, construct a new weakly toll walk $W'$, which we get by taking $W$ and leaving out any vertex that is not in $V(G)$. By the structure of the corona product graph we see, that $W'$ in fact is a walk and even a weakly toll walk between $g_1$ and $g_2$, which contains $g$ and all vertices of $W'$ are in $V(G)$, a contradiction.\qed
\end{proof}

\begin{lemma}\label{lema3}
Let $G$ and $H$ be connected non complete graphs and consider a corona product graph $G \circ H$. Let $i,j \in \lbrace 1,\ldots,n\rbrace$ be disctinct. Then  $WT_{G\circ H}(g_i,g_j)=WT_{G}(g_i,g_j) \cup \bigcup _{k \in A}H^k$, where $A=\lbrace k \in \lbrace 1,\ldots,n\rbrace ~|~ k\neq i, k\neq j, g_k\in WT_{G}(g_i,g_j) \rbrace$.
\end{lemma}

\begin{proof}
First, assume $g_ig_j \in E(G)$. Then $W:g_i,g_j$ is (the only) weakly toll walk between vertices $g_i$ and $g_j$ in $G \circ H,$ $A=\emptyset$ and $WT_{G\circ H}(g_i,g_j)=WT_{G}(g_i,g_j)$.  Therefore, assume from now on that $g_ig_j \notin E(G)$. 

Let $g_k \in V(G)\setminus \lbrace g_i,g_j \rbrace$. By Proposition \ref{sprehodi}, $g_k \in WT_{G\circ H}(g_i,g_j)$ if and only if $g_k \in WT_{G}(g_i,g_j)$.  First, let $g_k \in WT_{G}(g_i,g_j)$,  let $W_k:g_i,\ldots,g_k,\ldots,g_j$ be a weakly toll walk in $G$ and $Q_k$ a walk in $H^k$, which contains all vertices of $H^k$. Then $W:g_i,\ldots,g_k,Q_k,g_k,\ldots,g_j$ is a weakly toll walk in $G \circ H$ that contains every vertex of $H^k$.

To conclude the proof, observe that for any $k \in \lbrace 1,\ldots,m\rbrace$ $h_k^i \notin WT_{G\circ H}(g_i,g_j)$ because $h_k^ig_l \notin E(G \circ H)$ for any $l \neq i$. Similarly, $h_k^j \notin WT_{G\circ H}(g_i,g_j)$ for any $k \in \lbrace 1,\ldots,m\rbrace$.  
If $g_k \notin WT_{G}(g_i,g_j)$ also $h^k_l \notin WT_{G\circ H}(g_i,g_j)$ for any $l \in \lbrace 1,\ldots,m\rbrace$.\qed
\end{proof}

To be able to describe weakly toll intervals for the last case, we need the following definition:

\begin{definition}
Let $u$ and $v$ be two different, non-adjacent vertices of a graph $G$. A {\it semi weakly toll walk} $\tilde{W}$ between $u$ and $v$ in $G$ is a sequence of vertices of the form $$\tilde{W}: u=w_0,w_1,\ldots,w_{k-1},w_k=v,$$ where $k\ge 0$, which enjoys the following conditions for $k>0$:
\begin{itemize}
\item $w_iw_{i+1}\in E(G)$ for all $i \in \lbrace 0,\ldots,k-1\rbrace$,
\item $uw_i\in E(G)$ implies $w_i=w_1$, where $i\in \lbrace 1,\ldots , k \rbrace$.
\end{itemize}
A semi wekly toll interval is the set $$SWT_G(u,v)=\lbrace x \in V(G) ~|~ x \text{ lies on a semi weakly toll walk between $u$ and $v$.} \rbrace.$$
\end{definition}

Note that it is important which vertex is the beginning of a semi weakly toll walk and which the end (which was irrelevant in the definition of the weakly toll walk.) The definition states that the first vertex of the semi weakly toll walk ($u$) has the same restrictions as end-vertices in the definition of the weakly toll walks, while there are no restrictions for the last vertex of the semi weakly toll walk ($v$). It can also appear more than once in the walk.

\begin{lemma}\label{lema4}
Let $G$ and $H$ be connected non complete graphs and consider a corona product graph $G \circ H$. Let $i,j \in \lbrace 1,\ldots,n\rbrace$,  $k \in \lbrace 1,\ldots,m\rbrace$ and consider vertices $g_i,h^j_k$. Let $S=SWT_G(g_i,g_j)$.
\begin{equation*}
WT_{G\circ H}(g_i,h^j_k) = \left\{
\begin{array}{cl}
\lbrace g_i,h_k^i \rbrace; & \text{if } i = j,\\
\lbrace g_i,g_j \rbrace \cup \left(V(H^j) \setminus N_{H^j}(h_k^j)\right) \cup \bigcup_{x \in S \setminus \lbrace g_i,g_j \rbrace}\left(\lbrace x \rbrace \cup V(H^x)\right); & \text{if } i \neq j.
\end{array} \right.
\end{equation*}	
\end{lemma}

\begin{proof}
If $i=j$, then $g_ih_k^i \in E(G \circ H)$ and $WT_{G\circ H}(g_i,h^j_k)=\lbrace g_i,h_k^i \rbrace$. 

Let now $i \neq j$.
By the definition of the corona product graph, it is clear that no vertex of $H^i$ belongs to  $WT_{G\circ H}(g_i,h^j_k).$ As any weakly toll walk between $g_i$ and $h^j_k$ contains a vertex $g_j$, 
it follows that if for some $l \in \lbrace 1,\ldots,m\rbrace$, $h_l^j h_k^j \in E(H^j)$, then  $h_l^j \notin WT_{G\circ H}(g_i,h^j_k).$ Let $P$ be a path from $g_i$ to $g_j$ in $G$ and $x \in V(H^k) \setminus N_{H^j}(h_k^j)$. Then $W:P,x,g_j,h_k^j$ is a weakly toll walk containing a vertex $x$. Note that if $x=h_k^j$, then $W:g_i,P,g_j,h_k^j$ is a corresponding weakly toll walk in $G \circ H$.

Let $g_s \in SWT_G(g_i,g_j)$. Then there is a semi weakly toll walk $S'$ between $g_i$ and $g_j$ in $G$ that contains a vertex $g_s$. Denote with $S_{g_i,g_s}$ a subwalk of $S'$ between $g_i$ and $g_s$, $S_{g_s,g_j}$ a subwalk of $S'$ between $g_s$ and $g_j$ and $Q^s$ a walk in $H^s$, which goes throught every vertex of $H^s$.
Then $W:S_{g_i,g_s},Q^s,S_{g_s,g_j},h_k^j$ is a weakly toll walk in $G \circ H$ containing vertices $g_s$ and every vertex of $V(H^j)$.

To conclude the proof, assume that $g_s \notin SWT_G(g_i,g_j)$. As every toll walk, that contains a vertex $h_l^s$ for some $l \in \lbrace 1,\ldots,m\rbrace$, must also contain a vertex $g_s$, no vertex of $V(H^s)$ belongs to $WT_{G\circ H}(g_i,h^j_k)$.\qed
\end{proof}

The description of weakly toll intervals in the corona product graphs gives the following results on the weakly toll number of the corona product graphs. By Lemma \ref{lema1} it immidiately follows that $wtn(G \circ H)=2$ if $wtn(H)=2$.

\begin{corollary}\label{navadna2}
Let $G$ and $H$ be connected graphs with $wtn(H)=2$. Then $wtn(G \circ H)=2.$
\end{corollary}

\begin{proof}
Let $W=\lbrace h_1,h_2 \rbrace$ be a weakly toll set of $H$. By Lemma \ref{lema1}, a set $\lbrace h_1^1,h_2^1 \rbrace$ is a weakly toll set of $G \circ H$ as the case $x \notin WT_{H^1}(h_1^1,h_2^1)$ cannot happen.\qed
\end{proof}

Even more, one additional vertex gives an upper bound for arbitrary connected graph $H$.

\begin{proposition}\label{zgornja3}
Let $G$ and $H$ be connected graphs and $H$ non complete graph. Then $wtn(G \circ H)\leq 3.$
\end{proposition}

\begin{proof}
Let $h_1^1,h_2^1$ be two non adjacent vertices of $H^1.$ Note that if $\lbrace h_1^1,h_2^1 \rbrace$ is a weakly toll set of $H^1$, by Corollary \ref{navadna2}, $wtn(G \circ H)=2$. So, assume that $\lbrace h_1^1,h_2^1 \rbrace$ is not a weakly toll set of $H^1$. By Lemma \ref{lema1},  $WT_{G\circ H}(h_1^1,h_2^1)=V(G \circ H) \setminus X$,  where $$X=\lbrace x \in V(H^i) ~|~ x\notin WT_{H^1}(h_1^i,h_2^1) \text{ and } x \text{ is adjacent to exactly one of } h_1^i \text{ and } h_2^i\rbrace.$$
Now, consider a set $\lbrace h_1^1,h_2^1,g_2 \rbrace$, where $g_1g_2 \in E(G)$. Let $x \in V(H^1)$ such that $xh_1^1 \in E(H^1)$ and $xh_2^1 \notin E(H^1)$. Then $W: g_2,g_1,x,g_1,h_2^1$ is a weakly toll walk in $G \circ H$ that contains $x$. If $x \in V(H^1)$ is such a vertex that $xh_1^1 \notin E(H^1)$ and $xh_2^1 \in E(H^1)$, then $W': g_2,g_1,x,g_1,h_1^1$ is a weakly toll walk in $G \circ H$ that contains $x$. Therefore $\lbrace h_1^1,h_2^1,g_2 \rbrace$ is a weakly toll set and  $wtn(G \circ H)\leq 3.$
\qed
\end{proof}

Even more, investigating Lemma \ref{lema1}, Lemma \ref{lema2}, Lemma \ref{lema3} and Lemma \ref{lema4} we observe that for given graphs $G$ and $H$, maximum weakly toll intervals arrise when the weakly toll interval is taken between two non adjacent vertices of one copy of $H$ (as described in Lemma \ref{lema1}).
Such a weakly toll interval is maximum iff a set
$$X=\lbrace x \in V(H^i) ~|~ x\notin WT_{H^i}(h_1^i,h_2^i) \text{ and } x \text{ is adjacent to exactly one of } h_1^i \text{ and } h_2^i\rbrace.$$ is minimum, i.e. an empty set.  By the notation, used in Lemma \ref{clanek2}, observe that $X=X_{h^i_1} \cup X_{h^i_2}$. Moreover, applying Corolarry \ref{glavnaPOSL} leads to $X_{h^i_1} \cup X_{h^i_2}=\emptyset$ if and only if $wtn(H)=2.$

Therefore we have the following:

\begin{theorem}\label{glavniC}
Let $G$ and $H$ be connected non complete graphs. Then
\begin{equation*}
wtn(G \circ H) = \left\{
\begin{array}{cl}
2; &wtn(H)=2,\\
3; & \text{otherwise.}
\end{array} \right.
\end{equation*}	
\end{theorem}

Even more:
\begin{theorem}\label{hullCor}
Let $G$ and $H$ be connected non complete graphs. Then $wth(G\circ H) = 2.$
\end{theorem}

\begin{proof}
Note that $wth(G\circ H)\leq wtn(G\circ H) \leq 3.$ Recall the proof of the Proposition \ref{zgornja3}. Let $S=\lbrace h_1^1,h_2^1 \rbrace$. Then  $\left[ S \right]]_{WT}=V(G \circ H)$ as the vertex $g_2 \in WT(h_1^1,h_2^1)$.\qed
\end{proof}

In addition, recall the definition of the generalized corona product graph. Given simple graphs $G,H_1, \ldots , H_n$, where $n = |V (G)|$, the generalized corona, denoted $G\circ \wedge_{i=1}^n H_i$, is the graph, obtained by taking one copy of graphs $G,H_1, \ldots , H_n$, and joining the $i-th$ vertex of $G$ to every vertex of $H_i$. Summing up the results, obtained for the corona product graph, it follows immidiately that $wtn(G\circ \wedge_{i=1}^n H_i) \leq 3$ if at least one of $H_i$ is not a complete graph. Even more, if there is one of $H_i$, for which $wtn(H_i)=2$, then $wtn(G\circ \wedge_{i=1}^n H_i) =2$. Obviously, by Theroem \ref{hullCor}, $wth(G\circ \wedge_{i=1}^n H_i) =2$.

\section{Acknowledgements}

This research was supported by Slovenian Research Agency under the grants P1-0383, J1-2452 and J1-1693.


\begin{thebibliography}{99}

\bibitem{al-14} L.~Alc\'on, A note on path domination, Discuss.~Math.~Graph Theory. 36 (2016) 1021--1034.

 \bibitem{abg} L.~Alc\'on, B.~Bre\v{s}ar, T.~Gologranc, M.~Gutierrez, T.~Kraner, I.~Peterin, A.~Tepeh, Toll convexity, European J.~Combin. 46 (2015) 161--175.

\bibitem{bkh} B.~Bre\v sar, S.~Klav\v zar, A.~Tepeh Horvat, On the geodetic number and related metric sets in Cartesian product graphs, Discrete Math. 308 (2008) 5555--5561.

\bibitem{bkt} B.~Bre\v{s}ar, T.~Kraner \v{S}umenjak, A.~Tepeh, The geodetic number of the lexicographic product of graphs, Discrete Math. 311 (2011) 1693--1698.

\bibitem{chm} J.~C{\'a}ceres, C.~Hernando, M.~Mora, I.~M.~Pelayo, M.~L.~Puertas, On the geodetic and the hull numbers in strong product graphs, Comput.~Math.~Appl., 60 (11) (2010) 3020--3031.

\bibitem{CC} G. B.~Cagaanan, S. R.~Canoy Jr., On the hull sets and hull number of the Cartesian product of graphs, Discrete Math. 287 (2004) 141--144.

\bibitem{hull} 
M.~C.~Dourado, Computing the hull number in toll convexity,  Ann. Oper. Res. 315 (2022) 121--140.

\bibitem{weakly} 
M.~C.~Dourado, M.~Gutierrez, F.~Protti, S.~Tondato, Weakly toll convexity and proper interval graphs, arXiv:2203.17056v6.

\bibitem{zahtevnost} 
M.~C.~Dourado, M.~Gutierrez, F.~Protti, S.~Tondato, Computing the hull and interval numbers in the weakly toll convexity, arXiv:2303.07414v1.

\bibitem{tanja} 
T.~Dravec, On the toll number of a graph, Disc. Appl. Math 321 (2022) 250--257.

\bibitem{fj-86}
 M.~Farber and R.~E.~Jamison, Convexity in graphs and hypergraphs,
 SIAM J. Alg. Discrete Math. 7 (1986) 433--444.

\bibitem{TP1}
T. Gologranc, P. Repolusk, Toll number of the Cartesian and the lexicographic product of graphs, Discrete Math. 340 (2017) 2488--2498.

\bibitem{TP2}
T. Gologranc, P. Repolusk, Toll number of the strong product of graphs, Discrete Math. 342 (2019) 807--814.

\bibitem{ImKl} R. Hammack, W. Imrich, S. Klav\v{z}ar, \emph{Handbook of Product Graphs, Second Edition}, CRC Press, Boca Raton, FL, 2011.

\bibitem{Pela} I. M. Pelayo, {\it Geodesic Convexity in Graphs}, Springer, New York, 2013.


\bibitem{stg} A.~P.~Santhakumaran, P.~Titus, K.~Ganesamoorthy, On the monophonic number of a graph, J.~Appl.~Math. Informatics 32 (2014) 255--266.


\end{thebibliography}
\end{document}